\documentclass[12pt, A4]{article}
\usepackage{amssymb}
\usepackage{amsmath,amsfonts,amsthm,amstext}
\usepackage[english]{babel}
\usepackage[latin1]{inputenc}
\usepackage{makeidx}
\usepackage{amscd}
\usepackage[dvips]{graphicx}
\usepackage{fancybox}
\usepackage{niceframe}
\usepackage{times}
\usepackage{mathrsfs}
\usepackage{bbm}
\usepackage[all]{xy}
\usepackage{fancyhdr}
\usepackage{sectsty}
\usepackage[Conny]{fncychap}
\usepackage[symbol]{footmisc}
\usepackage[T1]{fontenc}
\DefineFNsymbols{stars}[text]{* {**} {*} {\S} {\I}}
\setfnsymbol{stars}

\newtheorem{thm}{Theorem}[section]
\newtheorem{cor}[thm]{Corollary}
\newtheorem{pro}[thm]{Proposition}
\newtheorem{deff}[thm]{Definition}
\newtheorem{lem}[thm]{Lemma}
\newtheorem{rem}[thm]{Remark}

\newcommand{\nc}{\newcommand}

\nc{\cc}{\D{C}} \nc{\hh}{\D{H}} \nc{\nn}{\D{N}} \nc{\oo}{\D{O}}
\nc{\qq}{\D{Q}}
 \nc{\rr}{\D{R}}
\nc{\zz}{\D{Z}} \nc{\livre}{\ast}

\nc{\barr}{\begin{array}} \nc{\earr}{\end{array}}
\nc{\bthm}{\begin{thm}} \nc{\ethm}{\end{thm}}
\nc{\bpro}{\begin{pro}} \nc{\epro}{\end{pro}}
\nc{\blem}{\begin{lem}} \nc{\elem}{\end{lem}}
\nc{\bins}{\begin{ins}} \nc{\eins}{\end{ins}}
\nc{\bcor}{\begin{cor}} \nc{\ecor}{\end{cor}}
\nc{\brem}{\begin{rem}} \nc{\erem}{\end{rem}}
\nc{\bdeff}{\begin{deff}} \nc{\edeff}{\end{deff}}
\nc{\bea}{\begin{eqnarray}} \nc{\eea}{\end{eqnarray}}
\nc{\D}[1]{{\mathbb#1}}
\def\R{\rm I\kern -.2em R}
\def\N{\rm I\kern -.18em N}
\def\Z{\rm Z\kern -.332em Z}
\def\de{\rm [\kern -.15em [}
\def\dd{\rm ]\kern -.15em ]}
\def\||{\hspace{0.15cm}|\hspace{0.15cm}}

\newcommand{\C}{\cal{C}}
\def\dbigcup{\mathinner{\bigcup \mkern -13.2mu \rlap{\raise 0.6ex\hbox{.}}\mkern 14.9mu}}
\title{Limit Groups are Subgroup Conjugacy Separable}
\author{S. C. Chagas,
 \,\,\, P. A. Zalesskii
 \footnote{\vspace*{-.5cm} Both authors were supported by
CNPq.}}
\begin{document}
%====================================================================

%\author{S. C. Chagas and P. A. Zalesskii}
%\address{Universidade de Brasília\\
%Campus Universitário Darcy Ribeiro\\
%70910-900 Brasília\\
%Brazil }

%\email{pz@mat.unb.br}

%\subjclass{MSC: 20E26; 20E18}

%\keywords{Subgroup Conjugacy Separability, Conjugacy Separability}

%\date{}
%----------additions
%\dedicatory{}
%%% ----------------------------------------------------------------------

\maketitle

\begin{abstract}
 A group $G$ is called subgroup conjugacy separable
 if for
every pair of non-conjugate finitely generated subgroups  of $G$,
 there exists a finite quotient of $G$ where
the images of these subgroups are not conjugate.
  We prove
that limit groups are subgroup conjugacy separable. We also prove
this property for one relator groups of the form $R=\langle
a_1,...,a_n\mid W^n\rangle$  with $n>|W|$. The property is also proved for virtual retracts (equivalently for quasiconvex subgroups) of hyperbolic virtually special groups. 
\end{abstract}

\section{Introduction}

O. Bogopolski and F. Grunewald \cite{BG} recently introduced the
important notion of subgroup conjugacy separability for a group
$G$. A group $G$ is said to be subgroup conjugacy separable if for
every pair of non-conjugate finitely generated subgroups $H$ and
$K$ of $G$,
 there exists a finite quotient of $G$ where
the images of these subgroups are not conjugate. They proved that
free groups and the fundamental groups of finite trees of finite
groups subject to a certain normalizer condition, are subgroup
conjugacy separable.  For finitely generated virtually free groups the result was proved in  \cite{CZ-15}. Also, O. Bogopolski and K-U. Bux in \cite{BB}
proved that surface groups are conjugacy subgroup separable.

Surface groups belong to the class of limit groups, the object of
extensive study in the last few decades due to the fact that they
play a key role in the solution of the Tarski problem.

%The notion in particular includes the definition of subgroup
%conjugacy separability for subgroup separable (LERF) groups.

Our main result generalizes the result of Bogopolski and Bux.

\begin{thm}\label{limit}
 Let $G$ be a limit  group.
 Then $G$ is subgroup conjugacy separable.
\end{thm}

 Bogopolski and Grunewald in their paper used also a notion of into conjugacy separability.  A
 subgroup $H$ of a group $G$ is called into conjugacy separable if for every finitely
 generated subgroup $K$ not conjugate into $H$ there exists a finite quotient
 of $G$ where the image of $K$ is not conjugate into the image of $H$. In this paper we do not need to ask for
 $K$ to be finitely generated. So changing slightly the definition of Bogopolski-Grunewald we say that
 a subgroup $H$ of a group $G$ is  into conjugacy distinguished if for  subgroup $K$ not conjugate into $H$ there
exists a finite quotient of $G$ where the image of $K$ is not
conjugate into the image of $H$. In terms of the profinite
completion it reads as follows: $H$ is into conjugacy
distinguished if every subgroup $K$ of $G$, the closure
$\overline K$ is conjugate  into $\overline H$ in $\widehat G$ if and only if $K$ is
conjugate into $H$ in $G$. We show in the paper that every
finitely generated subgroup of a limit group is into conjugacy
distinguished.

The methods of the proof are based on the paper \cite{RZ-15} of
 Ribes an the second author on groups whose finitely generated
subgroup are conjugacy distinguished. In particular, we use ideas
of  Section 3 from that paper, where the virtual retract property
plays a crucial role. 

This allows to extend Theorem \ref{limit} to virtual retracts of hyperbolic groups with conjugacy separable finite index subgroups.

\begin{thm} \label{virtualretract} Let $G$ be a hyperbolic  group such that every finite index subgroup of $G$ is conjugacy separable and let $H$
be a virtual retract of $G$. Then $H$ is into conjugacy
distinguished. In particular $G$ is a virtual retracts subgroup  conjugacy separable. \end{thm}

A group $G$ is called virtually  special if there exists a special
compact cube complex X having a finite index subgroup of $G$ as
its fundamental group (see \cite{W} for definition of special cube
complex). Virtually special groups own its importance to  Daniel
Wise who proved in \cite{W} that 1-relator groups with torsion are
virtually special, answering positively a question of Gilbert
Baumslag who asked in \cite{B-67} whether this groups are
residually finite. In fact, many  groups of geometric origin are
virtually special: the fundamental group of a hyperbolic
3-manifold (Agol \cite{A-13}), small cancellation groups
(a combination of \cite{W} and \cite{A-13}) and hyperbolic Coxeter
groups (Haglund and Wise \cite{HW-2010}) are virtually
 special.

 Moreover,  Haglung and Wise in \cite{HW-2008} showed that quasiconvex subgroups of 
a virtually  special hyperbolic group $G$ (i.e.,
 a subgroups that represents  a quasiconvex subset in the  set of vertices of the Cayley graph of
$G$) are virtual retracts of $G$. Thus the next theorem applies in
particular to this important class of subgroups.

\begin{thm} \label{quasiconvex} Let $G$ be a hyperbolic virtually special group and let $H$
be a quasiconvex subgroup of $G$. Then $H$ is into conjugacy
distinguished. In particular $G$ is quasiconvex subgroup conjugacy separable. \end{thm}

As an application of it we obtain

\begin{thm} Let $R=\langle a_1,...,a_n\mid W^n\rangle$ be a one relator group
with $n>|W|$. Then every finitely generated subgroup $H$ of $R$ is
into conjugacy distinguished and $R$ is subgroup conjugacy
separable.\end{thm}

After this paper  was submitted Bogopolski and Bux put  the paper \cite{BB} into arxiv, where they gave independent prove of Theorems \ref{virtualretract} and \ref{quasiconvex} in \cite[Lemma 6.2  and Corollary C]{BB} for torsion free groups. Our methods allow us to avoid the assumption of torsion freeness. In particular, the case  of small cancelation groups groups with finite $C'(1/6)$  or $C '(1/4) - T(4)$ presentations is covered by our results.

We finish the paper showing that a direct product of two free
groups is not subgroup conjugacy separable (see Section 3).

\begin{pro} A direct product $F_2\times F_2$ of two free groups of
rank $2$ is not subgroup conjugacy separable.\end{pro}

\bigskip
{\it Acknowledgements:} We thank Ashot Minasyan for conversations
on the subject.

\section{Proofs}

A subgroup $H$ of a group $G$ is called a virtual retract if $H$
is a semidirect factor (retract) of some finite index subgroup of $G$.  A group $G$ is called hereditarily conjugacy separable if every finite index subgroup of $G$ is conjugacy separable.

\begin{lem}  \label{achieve} Let $G$ be a group, $U$ a finite index conjugacy separable subgroup of $G$ and $H$ is a rectract of $U$.  Let $K$ be a subgroup of $G$ such that $\overline K^\gamma\leq \overline H$ for some $\gamma\in \widehat G$  and  $k$ be an element of $K$. Then there are $G$-conjugates $H'$ and $K'$ of $H$ and $K$  respectively, such that $(K')^{\gamma'}\leq \overline H'$ for some $\gamma'\in C_{\overline U}(k)$. Moreover, if  $C_G(k)$ is virtually cyclic and $U$ is hereditarily conjugacy separable, then the $\gamma'$  can be achieved to be in $\overline{\langle k\rangle}$. \end{lem}

\begin{proof}  We shall replace $H,K$ by their conjugates in $G$ and change $\gamma$ correspondingly until we achieve  the statement of the lemma holding for them; thus the final $H$ and $K$ will be our $H'$ and $L'$ of the statement. 

 Note first that $G$ is residually finite, since $U$ is, so we can regard $G$ as a dense subgroup of $\widehat G$ . Since $G\widehat{U}= \widehat{G}$, replacing $K$ by some
conjugate in $G$ we may assume that $\gamma$ belongs to
$\widehat{U}$ and $K$ is contained in $U$, since $\overline{H}\leq
\overline{U}$ and $U$ is closed in the profinite topology. 
By Proposition 7 in \cite{RZ-15} $H$ is conjugacy
distinguished, therefore $k$ is conjugated to an element of $H$ in
$U$. Hence, we may assume that $k$ belongs to $H$.

Let $f: U \rightarrow H$ be the epimorphism with the restriction
to $H$ being the identity map and $\hat{f}:
\widehat{U}\rightarrow \overline{H}$ be the continuous extension
of it.

We have, $k^{\gamma}\in \overline{H}$, so $k^{\gamma}=
\widehat{f}(k^{\gamma})=f(k)^{\widehat{f}(\gamma)}=k^{\widehat{f}(\gamma)}\in
\widehat{f}(\overline{H})= \overline{H}$. Hence,
$\widehat{f}(\gamma)^{-1}\gamma\in C_{\widehat{G}}(k)$.

Replacing $\gamma$ by $\widehat{f}(\gamma)^{-1}\gamma$ we achieve
 that $\gamma$ centralizes $k$.

Note however that if $C_G(k)$ is virtually cyclic, the group generated by
$k$ has finite index in $C_U(k)$, and since $U$ is hereditarily conjugacy separable by Proposition 3.2
in \cite{M}  $C_{U}(k)$ is dense in $C_{\widehat{U}}(k)$. Hence,

$$\widehat{\langle h \rangle} C_{U}(k)= C_{\widehat{U}}(k).$$

It means that conjugating $K$ by an element of $C_U(k)$, we may
assume that $\gamma\in \widehat{\langle k \rangle}$.
\end{proof}

 Our
main tool is the following proposition whose proof uses
essentially Proposition 7 in \cite{RZ-15}.

\begin{pro}\label{general} Let $G$  be a hereditarily conjugacy separable group and  $H$ be a virtual retract of $G$.  Let
$K$ be a subgroup of $G$ having an element $h$ such that $C_G(h)$
is virtually cyclic.  Then $\overline K$ is conjugate  into $\overline H$ in
$\widehat G$ if and only if $K$ is conjugate  into $H$ in $G$.
 Moreover, if $K$ is closed then $\overline K$ is conjugate to
$\overline H$ in $\widehat G$ if and only if $K$ is conjugate  to
$H$ in $G$.\end{pro}

\begin{proof} Suppose $\overline K^{\gamma}\leq \overline H$, where
$\gamma\in \widehat{G}$.

By hypothesis $H$ is a virtual retract of $G$. So there exist a
finite index subgroup $U$ of $G$ such that $H$ is a retract of
$U$.  Then by Lemma \ref{achieve} we may
assume that $\gamma\in \widehat{\langle h \rangle}$. This implies
that $K\leq \overline{H}$, and since by Corollary 3.1.6 (b)  \cite{RZ-10} $H$ is closed (i.e. $H= \overline{H}\cap G$) we
have $K\leq H$.

Assuming in addition that $\overline K^{\gamma}=\overline H$ and
$K$ is closed we have $H=\overline H\cap G=\overline K \cap G= K$
that shows the last statement of the proposition.
\end{proof}

\begin{cor}\label{virtually} Let $G$  be a hereditarily conjugacy separable group  and $H$
a finitely generated subgroup of $G$.  Let $U$ be a finite index
subgroup of $G$ such that $H\cap U$ is a  retract of $U$. Let $K$
be a subgroup of $G$ having an element  $h$ such that $C_G(h^n)$
is virtually cyclic for every natural $n$.  Then $\overline K$ is conjugate
into $\overline H$ in $\widehat G$ if and only if $K$ is conjugate
into $H$ in $G$.
 Moreover, if $K$ is closed then $\overline K$ is conjugate to
$\overline H$ in $\widehat G$ if and only if $K$ is conjugate  to
$H$ in $G$. \end{cor}

\begin{proof} Suppose $\overline K^{\gamma}\leq \overline H$, where
$\gamma\in \widehat{G}$.   By Proposition \ref{general}
$K\cap U$ is conjugate into $H\cap U$ in $U$ so we may assume that
$K\cap U \leq H\cap U$. Choose natural $n$ such that $h^n\in K\cap
U$.  By Lemma \ref{achieve} we may
assume that $\gamma\in \widehat{\langle h^n \rangle}$. This
implies that $K\leq \overline{H}$, and since  $H= \overline{H}\cap
G$  (indeed, $H\cap U$ is a retract of $U$,  hence  is closed and so $H$ is
closed), we have $K\leq H$.

Assuming in addition that $\overline K^{\gamma}=\overline H$ and
$K$ is closed we have $H=\overline H\cap G=\overline K \cap G= K$
that shows the last statement of the corollary.\end{proof}

 To apply Proposition \ref{general} to limit groups  we shall need
 the following easy

\begin{lem}\label{lem} Let $G$ be a finitely generated  non-abelian limit
group. Then $G$ has an element whose  centralizer is cyclic.
 \end{lem}

 \begin{proof} Let $G_n=G_{n-1}*_C A$ be n-th  extension of
 centralizers ($A$ is free abelian of rank $m$) such that $G\leq G_n$.   Let $a,b\in G$ be non-commuting elements of $G$.  Since $G_n$ is commutative transitive,  the centralizer of any
 element of $G_n$ is free abelian and if it is non-cyclic it  must intersect a conjugate of $C$ by  Theorem 14 \cite{Serre}  and so must be 
 conjugate to $A$.  We need to find an element  in $G$ not conjugate to an element of $A$.  Therefore we may assume that $a\in A^g, b\in A^h$, $A^g\neq A^h$ for some $g,h\in G_n$ and in fact conjugating $G$ by $g^{-1}$ we may assume that $a\in A$. It follows then from the canonical normal form of $ab$ in $G_n$ that it can not be conjugate to an element  of $A$ in $G_n$.\end{proof}

\begin{thm}
 Let $G$ be a  limit group. Then $G$ is  subgroup conjugacy
 separable. Moreover, every finitely generated subgroup of $G$ is  into conjugacy
 separable.
\end{thm}
\begin{proof} Note first that $G$ is hereditarily conjugacy separable (see
Proposition 3.8 in \cite{CZ-07}). Let $H$ be a finitely generated subgroup of $G$.
By Theorem B \cite{W-08} $H$ is a virtual retract of $G$. Let $K$
be a finitely generated subgroup such that
 $\overline K^{\gamma}\leq \overline H$ for some
 $\gamma\in \widehat{G}$.

We distinguish two cases.

\medskip\noindent 1. $K$ is not abelian.  Then by  Lemma \ref{lem} there exists an element $k\in K$ whose centralizer is cyclic and the result follows from  Proposition \ref{general}.

\medskip\noindent 2. $K$ is abelian. Let $k\neq 1$ be an element of
$K$. By Lemma \ref{achieve} we may assume that $K^{\gamma}\in \overline H$ for some $\gamma\in C_{\overline G}(k)$ and since  $K\leq C_G(k)\leq C_{\widehat
G}(k)$ by commutative transitivity property  we have $K^\gamma=K$ and so  $K\leq \overline H \cap G=H$ since $H$ is closed in $G$. If $\overline K^\gamma=\overline H$ then the last formula gives the equality $K=H$.
\end{proof}

Next we apply Corollary \ref{virtually}  to important groups  of
geometric nature.

\begin{thm} \label{virtualretractsection} Let $G$ be a heredetarily conjugacy separable hyperbolic  group and let $H$
be a virtual retract of $G$. Then $H$ is into conjugacy
distinguished. In particular $G$ is a virtual retract subgroup  conjugacy separable. \end{thm}

\begin{proof}   Observing  that the centralizers
of elements of infinite order of $G$ are virtually cyclic (Proposition 3.5 \cite{Mih}) one deduces the
result from Corollary \ref{virtually}.\end{proof}

\noindent
{\bf Proof of Theorem \ref{quasiconvex}.}  By Theorem 1.1 in  \cite{MZ-15} $G$ is hereditarily conjugacy separable and by \cite{HW-2008} quasiconvex subgroups of $G$ are virtual retracts.  So the result follows from Theorem \ref{virtualretractsection}.

\begin{thm} Let $R=\langle a_1,...,a_n\mid W^n\rangle$ be a one relator group
with $n>|W|$. Then every finitely generated subgroup $H$ of $R$ is
into conjugacy distinguished and $R$ is subgroup conjugacy
separable.\end{thm}

\begin{proof} By   Theorem 1.4 in
\cite{W} $R$ is hyperbolic virtually special and so every  quasiconvex subgroup of it is a virtual retract  by Proposition 4.3 in
\cite{B}.  On the other hand by Theorem 1.2 in \cite{HW}  every finitely generated subgroup of $R$ is quasiconvex. Thus one deduces the result from the previous
theorem.\end{proof}

%has that every finitely generated subgroup of $R$ has a subgroup
%of finite index which is a virtual retract of $R$.
%Another important fact about $R$, proved by Newman in \cite{N},
%Theorem 2, states that the centralizers of nontrivial elements in
%one-relator groups with torsion are cyclic. On the other hand, by
%Theorem 1.1 in \cite\{MZ}, $R$ is hereditarily conjugacy
%separable.
%Therefore, using  we get the result.

\begin{rem} Let $\C$ be a class of finite groups closed for
subgroups, quotients and extensions.  One can define then subgroup
$\C$-conjugacy separability and prove the pro-$\C$ version of 
Proposition \ref{general}  using Proposition 7 in \cite{RZ-15} and
Theorem 4.2 in \cite{F} instead of Proposition 3.2 in
\cite{M}.\end{rem}

\section{Direct product }

We show here that a direct product of free groups of rank 2 is not
subgroup conjugacy separable. It
 is based on an idea of Michailova combined with observations of
 V. Metaftsis  and E. Raptis \cite{MR-08}.

Consider a finitely presented  group $H$ given by G. Higman  \cite{H-51}
$$H =
\langle x_1, x_2, x_3, x_4 \,|\, r_1, r_2, r_3, r_4\rangle,$$
where $r_1= x_2^{-1}x_1x_2x_1^{-1}, r_2= x_3^{-1}x_2x_3x_2^{-1},
r_3= x_4^{-1}x_3x_4x_3^{-1}, r_4= x_1^{-1}x_4x_1x_4^{-1}$ and let
$F_4$ be the free group on four generators $x_1, x_2, x_3, x_4$.

Clearly, $F_4 \times F_4'$  can be considered as a finite
index subgroup of $F_2\times F_2'$, where $F_2$ and $'F_4'$ are   isomorphic
copies of $F_2$ and $F'_4$ respectively.  Since the induced profinite topology on a finite
index subgroup is the full profinite topology,  a  subgroup of $F_4\times F_4'$ is closed in the
profinite topology of $F_4\times F_4$ if and only if it is closed
in the profinite topology of $F_2\times F_2'$.

Let $L_H$ be the subgroup of $F_4\times F_4$ generated by
$$L_H = \langle (x_i, x_i), (1, r_i), \,\, i = 1,2,3,4\rangle.$$

Then $L_H\cap (F_4\times \{1\})$ is the normal closure of $\langle r_1, r_2, r_3,
r_4\rangle$ in $(F_4\times \{1\})$. So, by Proposition 1 in \cite{MR-08} $L_H$
is closed in the profinite topology of $F_4\times F_4'$, if and
only if $L_H\cap (F_4\times \{1\})$ is closed in the profinite topology of $(F_4\times \{1\})$
or equivalently if and only if the group $H = \langle x_1, x_2,
x_3, x_4 \,|\, r_1, r_2, r_3, r_4\rangle$ is residually finite.
However, G. Higman in \cite{H-51} proved that $H$ possesses no
proper  normal subgroups of finite index, so $L_H$ is not closed
in the profinite topology of $F_4\times F_4'$.

Consider the closure $\overline{L_H}$ of $L_H$ in $F_4\times F_4'$.
Then $\overline{L_H}\cap (F_4\times \{1\})$ is a normal subgroup of
$F_4\times \{1\}$, and $(F_4\times \{1\})/(\overline{L_H}\cap (F_4\times \{1\}))$ is residually finite. On the other hand, $(F_4\times \{1\})/(\overline{L_H}\cap (F_4\times \{1\}))$ is a quotient of $H$, and
$H$ does not have any finite index normal subgroup, so $F_4\times \{1\}=\overline{L_H}\cap (F_4\times \{1\})$, consequently
$\overline{L_H}=F_4\times F_4'$. It means that $L_H$ is a dense
subgroup of $F_4\times F_4'$.

Now, we show that $F_4\times F_4'$ is not subgroup conjugacy
separable. Indeed, since $L_H$ is dense in $F_4\times F_4'$, its
image in each finite quotient coincides with the whole quotient.
If $L_H$ is not isomorphic to $F_4\times F_4'$, then is not
conjugate. Thus it suffices to prove that $L_H\not\cong F_4\times F_4'$.

The subgroup $L_H\cap (F_4\times \{1\})$ is an infinitely  generated
subgroup of $F_4\times F_4'$ (since it is normal of infinite index), and it is clear that the centralizer
of every element in $F_4\times F_4'$ is finitely generated. Let $(e,1)$ be a nontrivial element of $L_H\cap (F_4\times \{1\})$. 
Then
$$C_{L_H}(e, 1)= C_{L_H\cap (F_4\times \{1\})}(e)\times (L_H\cap (F_4\times \{1\})),$$
so  the right factor of  the centralizer is infinitely generated and hence $C_{L_H}(e, 1)$ is infinitely generated as well.
Therefore $L_H$ and $\overline{L_H}=F_4\times F_4'$ are not isomorphic subgroups.

\bigskip
{\it Author's Adresses:}

\medskip
Sheila C. Chagas\\
Departamento de Matem\'atica,\\
~Universidade de Bras\'\i lia,\\
70910-900 Bras\'\i lia DF,\\
Brazil

sheila@mat.unb.br

\medskip
Pavel A. Zalesski\\
Departamento de Matem\'atica,\\
~Universidade de Bras\'\i lia,\\
70910-900 Bras\'\i lia DF\\
Brazil

pz@mat.unb.br

\end{document}